\documentclass[11pt]{amsart}
\usepackage{amsmath, amssymb}
\usepackage{amsfonts}
\usepackage{mathrsfs}
\usepackage[arrow,matrix,curve,cmtip,ps]{xy}

\usepackage{amsthm}
\usepackage{setspace}
\allowdisplaybreaks

\newtheorem{theorem}{Theorem}[section]

\newtheorem*{theorem*}{Theorem}
\theoremstyle{remark}

\newtheorem{definition}[theorem]{Definition}
\newtheorem{example}[theorem]{Example}

\newcommand{\E}{\mathbb{E}}
\newcommand{\R}{\mathbb{R}}
\newcommand{\p}{\partial}

\numberwithin{equation}{section}

\begin{document}
\title{Information Geometry and Statistical Manifold}

\author{Mashbat Suzuki}

\address{Department of Mathematics \\McGill University \\ Montreal\\ Canada}
\email{mashbat.suzuki@mail.mcgill.ca}


\date{\today}


\begin{abstract}

We review basic notions in the field of information geometry such as Fisher metric on statistical manifold, $\alpha$-connection and corresponding curvature following Amari's work \cite{amari3, amari1, amari2}. We show application of information geometry to asymptotic statistical inference. 
\end{abstract}

\maketitle

\section{Geometric Structure of Statistical models}

We model a family of probability distribution functions given by a set of parameters as a Riemannian manifold. In other words, we consider each distribution as a point on a  Riemannian manifold. Consider a family $S$ of probability distribution functions over some space $\Omega$. Suppose that each distribution function may be parametrized using $n$ real valued variables $[\xi_1,\cdots,\xi_n]$ such that we can write
$$
S=\{ p_\xi=p(x;\xi) \ | \ \xi=[\xi_1,\cdots,\xi_n]\in \Xi \}
$$ 
where $p(x;\xi)$ is a probability distribution function on $\Omega$. We call $\Xi \subset \R^n$ a \textbf{parameter space}, it is chosen so that the map $\xi\rightarrow p_\xi$ is injective.  
 The set $S$ is called \textbf{parametric model}. It is in spirit analogous to the idea of fiber bundle over $\Omega$, where the fibers are distribution functions. 
 The main idea of information geometry is to use tools from Riemannian geometry to extract information from the underlying statistical model that is otherwise not easily accessable. The idea has been successfully used in various areas including statistical inference and manifold learning as shown in \cite{suna14} 
 In order to have more structure on $S$, we assume several regularity assumptions. We assume that $\Xi$ is open subset of $\R^n$, and for each $x_0\in \Omega$ the following map 
 \begin{align*}
 &\Xi\longrightarrow \R \\
 &\xi \longmapsto p(x_0;\xi)
 \end{align*}
is smooth. In addition we assume that the order of integration and differentiation can be freely re-arrenged. One last assumption on $S$ is that the support of each distribution function stays invariant under change of parameters. This assumption is taken in order to simplify the analysis.

We recognize the parametric model as a manifold by identifying equivalent coordinate charts, observe that the mapping $\phi:S\rightarrow \R^n$ given by $\phi(p_\xi)=\xi$ serves as a coordinate chart for $S$. Furthermore suppose we have $C^\infty$ diffeomorphism $\psi$ from $\Xi$ to $\psi(\Xi)$ then we can use $\eta=\psi(\xi)$ instead of $\xi$ as the parameter so that we obtain $S=\{p_{\psi^{-1}(\eta)} \ | \eta\in \psi(\Xi)\}$. Hence we can consider parametrizations that are related by $C^{\infty}$ diffeomorphisms to be equivalent which allows us to consider $S$ as smooth manifold. It is common to call $S$ satisfying above assumptions a statistical manifold. 

\section{Fisher Information Metric} 

Now that we have a manifold it is time to put a notion of distance on $S$, which means we need a way of measuring distance between probability distribution functions. The Fisher information metric measures closeness of the shape between two distribution functions, it is also proportional to the amount of information that the distribution function contains about the parameter. We define the following as the Fisher information matrix
$$
g_{ij}(\xi) \equiv \E_\xi[\p_i l_\xi \p_j l_\xi]
$$
where $l_\xi=\log p(x;\xi)$ and $\p_i=\frac{\p}{\p \xi^i }$. Note that it is possible to write $g_{ij}=-\E_\xi[\p_i\p_i l_\xi]$ since we can change the order of integration and differentiation. 

Observe that $g_{ij}(\xi)=g_{ji}(\xi)$, hence the Fisher information matrix is symmetric. In general $g_{ij}$ is positive semidefinite, since for any $n$-dimensional vector $c=[c_1,\cdots ,c^n]^t$ we get
\begin{align*}
c^t[g_{ij}(\xi)]c=c^i c^j g_{ij}(\xi) &= \sum_{i=1}^n \sum_{j=1}^n c_i c_j \E_\xi[\p_i l_\xi\p_j l_\xi] \\
&=\E_\xi\left[ \sum_{i=1}^n  c_i \p_i l_\xi \sum_{j=1}^n\p_j l_\xi \right] \\
&=\E_\xi\left[\left( \sum_{i=1}^n c^i\p_i l_\xi \right)^2\right] \\
&=\int_{\Omega} \left\lbrace \sum_{i=1}^{n}c^i\p_i l_\xi \right\rbrace^2 p(x;\xi)\geq 0
\end{align*}
We make $g_{ij}$ positive definite by assuming that the elements $\{ \p_1l_\xi,\cdots,\p_n l_\xi\}$ are linearly independent as functions, which by above equation makes $g_{ij}$ positive definite. This means that $g_{ij}$ provides Riemannian metric on $S$. Hence we have an inner product $g_\xi=\langle , \rangle_\xi$  on each tangent space $T_\xi S$, it can also be written as $g_\xi=g_{ij}dx^i\otimes dx^j$. In other words for any $X,Y\in T_\xi S$ we have
\begin{align*}
\langle X, Y\rangle_\xi &= g_{ij}dx^i\otimes dx^j(X^\mu \p_\mu, Y^\nu \p_\nu) \\
& =g_{ij}X^\mu \delta^i_\mu Y^\nu \delta^j_\nu\\
&=\E[X^iY^j\p_i l_\xi\p_j l_\xi] \\
& =\E[(Xl_\xi)(Y l_\xi)]
\end{align*}
We will give few examples of the metric and statistical manifolds. 
\begin{example} (Gaussian Model)

Consider $S=\{ N(\mu,\sigma^2)\}=\{ p(x;\mu,\sigma) \ | \ (\mu,\sigma)\in \R\times\R^+ \}$ where
$$
p(x;\mu,\sigma)=\frac{1}{\sigma\sqrt{2\pi}}\exp\left\lbrace -\frac{x-\mu}{2\sigma^2} \right\rbrace
$$
So in our case $\xi=[\mu,\sigma]$ and $\Xi=\R\times \R^+$, hence 
$$l_\xi=\log p(x;\xi)=-\log \sigma-\log \sqrt{2\pi}-\frac{(x-\mu)^2}{2\sigma^2}$$

Hence one can check that the metric in this case is 
$$
g_{ij}(\xi)=
\begin{bmatrix}
\frac{1}{\sigma^2} & 0\\
0 & \frac{2}{\sigma^2}
\end{bmatrix}
$$
the line element is thus $ds^2=\frac{d\mu^2+2 d\sigma^2}{\sigma^2}$.  This is similar to the Poincare metric on the upper half plane which is $ds^2=\frac{dx^2+dy^2}{y^2}$, this is constant negative one curvature metric, our statistical manifold is constant negative two curvature. The geodesics on $S$ are given by vertical straight lines and semi-ellipses.

\end{example}

We shall give an inequality that is used often in statistics using the metric. Let $\hat{\xi}$ be an estimator, of $\xi$ which means it is a function of the observed data $x$ which are realizations of some underlying true distribution $p_0(x;\xi)$. An estimator is called \textbf{unbiased estimator} if 
$$
\E_\xi[\hat{\xi}(X)]=\xi \ \text{for} \ \forall{\xi}\in\Xi
$$
The mean square error of an unbiased estimator $\hat{\xi}$ is given by the variance covariance matrix 
$$v^{ij}_\xi=\E[(\hat{\xi}^i(X)-\xi^i)(\hat{\xi}^i(X)-\xi^j)]$$
The variance covariance matrix is the generalization of variance to higher dimensions.
\begin{theorem}(Cramer-Rao inequality)

The variance-covariance matrix $v^{ij}_\xi$ of an unbiased estimator $\hat{\xi}$ satisfies $v^{ij}_\xi\geq g^{ij}_\xi$
\end{theorem}

Above bound provides a fundamental limit on the variance of an estimator, in other words if an estimator $\hat{\xi}$ attains the lower bound then it is the best estimator in the sense that any other estimator will have higher or equal variance than $\hat{\xi}$. 

\section{The $\alpha$-connection}

We define a connection by specifying the connection coefficients first. Consider following function as a connection coefficient, at each point $\xi\in S$ we define
$$\Gamma^{(\alpha)}_{ij,k}\equiv \E_\xi\left[ \left(\p_i\p_j l_\xi+\frac{1-\alpha}{2}\p_i l_\xi \p_j l_\xi\right)(\p_k l_\xi)\right]$$
where $\alpha\in\R$. This quantity transforms in the usual way under coordinate transformations hence we have an affine connection $\nabla^{\alpha}$ on $S$ defined by 
$$
\langle\nabla^{(\alpha)}_{\p_i}\p_j,\partial_k\rangle_\xi=\Gamma^{(\alpha)}_{ij,k}
$$
where the inner product is given by the Fisher metric. The connection defined as above is called \textbf{$\alpha$-connection}. The alpha connection is useful in asymptotic theory of estimation as we shall see later on. There are three important values for alpha, since as we shall see that we can get any other connection as combinations of these three, these three values of $\alpha$ are $-1,0,1$. 

We can exchange between any two $\alpha, \beta\in \R$ connection coefficients by the following formula
$$
\Gamma^{(\beta)}_{ij,k}=\Gamma^{(\alpha)}_{ij,k}+\frac{\alpha-\beta}{2}T_{ijk}
$$ 
where $T_{ijk}$ is a covariant symmetric tensor defined as 
$$
T_{ijk} \equiv \E_\xi[\p_i l_\xi\p_j l_\xi\p_k l_\xi]
$$

Above formulas can be used to show that general $\alpha$-connection satisfy following two equality 
\begin{align*}
\nabla^{(\alpha)} &= (1-\alpha)\nabla^{(0)}+\alpha\nabla^{(1)} \\ 
&=\frac{1+\alpha}{2}\nabla^{(1)}+\frac{1-\alpha}{2}\nabla^{(-1)}
\end{align*}
hence it suffices to study only $\nabla^{(-1)}$ and $\nabla^{(1)}$. As we shall see these two connections have special properties later on. 

\begin{theorem}
The 0-connection is the Levi-Civita connection with respect to the Fisher metric 
\end{theorem}
\begin{proof}
 We have to check that the zero connection is metric compatable, and torsion free. First we check metric compatibility, which says that for any $X,Y,Z\in T S$ we have 
 $$
 Z\langle X,Y\rangle=\langle \nabla_Z X\rangle Y+\langle X,\nabla_Z Y\rangle
 $$
 Let $X=X^i\p_i$, $Y=Y^j \p_j$ and $Z=Z^k \p_k$, so  at each point $T_\xi S$ we get
 $$
 Z^k X^i Y^j\p_k g_{ij}=Z^k  X^i Y^j \langle \nabla_{\p_k}^{(0)}\p_i,\p_j  \rangle +Z^k X^i Y^j \langle  \p_i,\nabla_{\p_k}^{(0)}\p_j \rangle
 $$
Hence it is enough to show that 
$$
\p_k g_{ij}=\langle \nabla_{\p_k}^{(0)}\p_i,\p_j  \rangle+ \langle  \p_i,\nabla_{\p_k}^{(0)}\p_j \rangle
$$
We check above by recalling the definition of $g_{ij}$ as an expectation
\begin{align*}
\p_k g_{ij} &=\p_k \E_\xi [\p_i l_\xi \p_j l_\xi] \\
&= \E[(\p_k \p_i l_\xi)(\p_j l_\xi)]+\E[(\p_i l_\xi)(\p_k\p_j l_\xi)]+\E[(\p_i l_\xi)(\p_j l_\xi)(\p_k l_\xi)]\\
&= \E[(\p_k \p_i l_\xi)(\p_j l_\xi) + \frac{1}{2}(\p_i l_\xi)(\p_j l_\xi)(\p_k l_\xi)]+ \\
&+\E[(\p_i l_\xi)(\p_k\p_j l_\xi)+\frac{1}{2}(\p_i l_\xi)(\p_j l_\xi)(\p_k l_\xi)] \\
&=  \Gamma^{(\alpha)}_{ki,j}+\Gamma^{(\alpha)}_{kj,i} \\
&= \langle \nabla_{\p_k}^{(0)}\p_i,\p_j  \rangle+ \langle  \p_i,\nabla_{\p_k}^{(0)}\p_j \rangle
\end{align*}
Hence the 0-connection is metric compatible. We now have to show that the connection is torsion free, but it is trivial since we know from differential geometry,  that connection is torsion free if and only if the connection coefficients are symmetric on the lower indicies. Since partial derivatives commute we have 
$$
\Gamma^{(\alpha)}_{ij,k}=\Gamma^{(\alpha)}_{ji,k}
$$
Therefore 0-connection is a Levi Civita connection with respect to the Fisher metric. 

\end{proof}

There are two common types of families that play an important role in information geometry and statistics. 
\begin{definition}
Ab $n$-dimensional model $S$ is called \textbf{exponential family}  if $S=\{p(x;\xi)\ | \ \xi\in \Xi\}$ with
$$p(x;\xi)=\exp\left\lbrace K(x)+\sum_{i=1}^{n}\xi^i F_i(x)-\psi(\xi) \right\rbrace$$
where $\{F_1,F_2,\cdots,F_n\}$ are  linearly independent  functions which are not identically constant on $\Omega$, and $\psi$ is a function on $\Xi$
\end{definition}
From the normalization condition we get the following formula 
$$
\psi(\xi)=\log\int_\Omega \exp\left[ K(x)+\sum_{i=1}^{n}\xi^i F_i(x)\right] dx
$$  
thus $\psi$ is determined by the choice of $K(x)$ and $F_i(x)$.
Many important models are known to be exponential families such as Normal, Multivariate Normal and Poisson distributions. There is a connection that is flat over exponential family this connection is $1$-connection for this reason we call $\nabla^{(1)}$ an \textbf{exponential connection} and denote it as $\nabla^{(e)}$. There is another type of family which is very common appear in statistics. 

\begin{definition}
Ab $n$-dimensional model $S$ is called \textbf{mixture family}  if $S=\{p(x;\xi)\ | \ \xi\in \Xi\}$ with
$$p(x;\xi)= K(x)+\sum_{i=1}^{n}\xi^i F_i(x)$$
where $\{F_1,F_2,\cdots,F_n\}$ are  linearly independent  functions which are not constant on $\Omega$
\end{definition}
The mixture family can be used model uniform, $F$-distribution and many other distributions. Similar to previous example there is a connection that is flat on mixture family, this connection is $\nabla^{(-1)}$  for this reason denote $\nabla^{(-1)}$ as $\nabla^{(m)}$. We summarize the information given here in the following theorem.
\begin{theorem}
An exponential family is e-flat and mixture family is m-flat
\end{theorem}

The usefulness of above theorem can be demonstrated in next theorem, reader can refer to \cite{amari1}, for more detailed discussion. Note that bias corrected estimator $\hat{u}^{*}$ is defined as
$$
\E[\hat{u}^{*}]-u=O\left( \frac{1}{N^2}\right)
$$
In other words, bias corrected estimators deviate from being unbias by a small amount as the number of trials increase. The following theorem gives an application of information geometry to asymptotic statistical inference and uses the ideas we have developed above.
\begin{theorem}
The mean square error of bias corrected first-order efficient estimator is give asymptotically by the expansion:
$$
\E[(\hat{u}^{*a}-u^a)(\hat{u}^{*b}-u^b)]=\frac{1}{N}g^{ab}+\frac{1}{2N^2}K^{ab}+O\left(\frac{1}{N^3}\right)
$$
where
 $$K^{ab}=(\Gamma^{(m)}_M)^{2ab}+2(H^{(e)}_M)^{2ab}+(H^{(m)}_A)^{2ab}$$
 where the terms $\Gamma^{(m)}_M,H^{(e)}_M$ and $H^{(m)}_A$ represent the m-connection coefficients of M, the embedding e-curvature of the model M, and the embedding m curvature of $A(u)$ which is estimating submanifold, respectively:
 \begin{align*}
 (\Gamma^{(m)}_M)^{2ab} &={\Gamma^{(m)}}^a_{cd}{\Gamma^{(m)}}^b_{ef}g^{ce}g^{df}\\
 (H^{(e)}_M)^{2ab} &={H^{(e)}}^k_{ce}{H^{(e)}}^\lambda_{df}g_{k\lambda}g^{cd}g^{ea}g^{fb}  \\
 (H^{(m)}_A)^{2ab} &= {H^{(m)}}^{a}_{k\lambda}{H^{(m)}}^{b}_{\mu\nu}g^{k\mu}g^{\lambda \nu}\\
 \end{align*}
\end{theorem}
Note that  $H^{(e)}_M$ is zero when $M$ is exponential family by the theorem (3.4). 
In order to get better asymptotic mean square error one can minimize individual curvatures by changing the estimator appropriately which decreases the mean square error . 
Hence information geometry gives methods to calculate asymptotic mean square errors more efficiently, and provides a way of designing estimators with desirable asymptotic properties. 
\begin{flushleft}

\textbf{Acknowledgement}: I would like to thank my project supervisor Prof. Prekash Panangedan, for his patience, guidance, and encouragement. I am very lucky to have had the opportunity to work with him and learn from him. 

\end{flushleft}
\bibliographystyle{plain}
\bibliography{BibMash}
\end{document}